\documentclass[11pt]{amsart}

\title{A general encryption scheme using two-sided multiplications with its cryptanalysis  }

\author{Vitali\u\i\ Roman'kov}
\address{Institute of Mathematics and Information Technologies\\Dostoevsky Omsk State  University}
\curraddr{}
\email{romankov48@mail.ru}

\usepackage{amsmath,amsthm,amsfonts,amssymb}
\usepackage{graphicx}
\usepackage{hyperref}
\usepackage[usenames]{color}

\newtheorem{theorem}{Theorem}[section]
\newtheorem{lemma}[theorem]{Lemma}
\theoremstyle{definition}

\newcounter{comcount}

\date{}

\begin{document}

\maketitle

\begin{abstract}
We show that many known schemes of the  public key exchange protocols in the algebraic cryptography, that use two-sided multiplications, are the specific cases of the general  scheme of such type. In  most cases,  such schemes are built on platforms that are subsets of the linear spaces. They have been repeatedly compromised by the linear decomposition method  introduced by the first  author.   The method allows to compute the exchanged keys without computing the private data and therefore without  solving the algorithmic problems on which the assumptions are based.  We demonstrate that this method can be successfully applied to the general scheme, thus it is in some sense universal. 
\end{abstract}

\section{Introduction}
\label{se:intro}

We show that many known schemes of the  public key exchange protocols in the algebraic cryptography, that use two-sided multiplications, are the specific cases of the general  scheme of such type. In  most cases,  such schemes are built on platforms that are subsets of  linear spaces, then they can be  compromised by the linear decomposition method  introduced by the first  author.   See for instance the  monograph
\cite{RomAC}, and papers  \cite{Roman}, \cite{MR}, \cite{RomMen}, \cite{GKR}, \cite{Romshield}.  The method allows to compute the exchanged keys without computing the private data and therefore without  solving the algorithmic problems on which the assumptions are based.  We demonstrate that this method can be successfully applied to the general scheme, introduced in this paper, thus it is in some sense universal. 

Some of the schemes under investigation were proposed by Andrecut \cite{A},  Wang et al.  \cite{WW},   Stickel \cite{S}, B. and T. Harley \cite{HH}, \cite{H}, Shpilrain and Ushakov  \cite{SU}.  Some other schemes were described in \cite{MSU1} and   \cite{MSU2}. The schemes using conjugation, for example the well-known scheme by Ko, Lee and al. \cite{KL}, that is a non-commutative version of the classical Diffie-Hellman scheme (see \cite{Romintro}), can be treated as schemes of the investigated type too. 

The general scheme proceeds as follows. Let $G$ be an  algebraic system with associative multiplication, for example group,  chosen as the platform. Further in the paper, $G$ is a group. We assume that $G$ is a subset of a finitely dimensional linear space $V$. Firstly, a set of public elements   $g_1, ..., g_k\in G$ is established. Then the correspondents, Alice and Bob, publish sequentially elements of the form $\phi_{a,b}(f)=afb; a, b \in G,$ where  $f\in G$ is   a given or previously built element. The parameters $a,b$ are private. The exchanged key has the form  
\begin{equation}
\label{eq:1}
K=\phi_{a_l, b_l}(\phi_{a_{l-1},b_{l-1}}( ... (\phi_{a_1,b_1}(g_i)...)) = a_la_{l-1} ... a_1g_ib_1 ... b_{l-1}b_l.
\end{equation} 
We suppose that Alice chooses parameters $a, b$  in a given finitely generated subgroup $A$ of $G$, and Bob picks  up parameters $a, b$ in a finitely generated subgroup $B$ of $G$ to construct their transformations of the form $\phi_{a,b}.$  Then, under some natural assumptions about $G, A$ and $B,$ we show that  each intruder can efficiently calculate the exchanged key $K$ without calculation the transformations used in the scheme.

Foundations of the linear decomposition method can be found in \cite{RomAC} (see also \cite{Roman} and \cite {MR}).   It can be applied only when the platform group $G$ is a part of a finite-dimensional linear space $V$. For example, $G\leq $ GL$_n(\mathbb{F})$ is a linear group of size $n$ that is a subset of the full matrix algebra M$_n(\mathbb{F})$ over a field $\mathbb{F}.$  In  many other cases when the platform group $G$ is not linear, or $G$ is linear but the dimension of $V$ is too large,
the linear decomposition method can be changed by the non-linear decomposition method (see \cite{Romnon}). It works when the platform group $G$ is a finitely generated nilpotent group or, more generally, polycyclic group. The class of polycyclic groups has been many times proposed as a source of good platforms for cryptographic schemes and protocols. See \cite{EK}, \cite{GK}, \cite{Pol}.

The transformations under discussion satisfy the equalities  $\phi_{a,b}\circ \phi_{c,d} = \phi_{ca,bd}$ and include  the unit element  $\epsilon = \phi_{1,1}$. The inverse of $\phi_{a,b}$ is $ \phi_{a^{-1}, b^{-1}}.$ 

Usually in the considering schemes, there are two finitely generated subgroups $A, B \subseteq G$, where the correspondences pick up their parameters. Alice chooses randomly elements in $A$, and Bob picks up randomly parameters in $B.$  In many schemes all elements of the subgroups  $A$ and $B$  are assumed  to be pairwise commuting, i.e., for every $a\in A$ and every $b\in B$ we have $ab = ba.$  Then all the transformations generate the subgroup $S(G,A,B)$ of $G.$ 

To get the exchanged key $K$ in (\ref{eq:1}) we need to apply to element 
 $g_i$  the transformation  $\phi_{u,v},$ where  $u = a_la_{l-1} ... a_1, v =  b_1 ... b_{l-1}b_l.$ We assume that all the transformations  $\phi_{a_j, b_j}$  have been used in the public data. In other words, for every pair  $a_j, b_j$ the list of public data contains  $c$ and  $d$  such that  $d = \phi_{a_j,b_j}(c).$ Then also  $c = \phi_{a_j,b_j}^{-1}(d),$ thus the transformation $\phi_{a_j,b_j}^{-1}$ can be considered as used too.  Alice and Bob compute $K$ using their private data. Alice knows a part of all transformations $\phi_{a_j,b_j},$ and Bob knows the complement of this part. In the follows section we'll show how we can efficiently compute $K$ under some natural assumptions on $G, A$ and $B$ without knowing of any of the private data. Note, that we don't compute any pair of elements $a_j, b_j.$  It turns out  that it doesn't matter who of the correspondents chooses concrete private elements $a_j, b_j$,  parameters of the transformation  $\phi_{a_j,b_j}$.  
Sometimes some schemes use  sums of elements $\varphi_{a,b}(f)$ (see \cite{H} and \cite{HH}). A variation of the method of linear decomposition allows to compute the exchanged key in these cases too (see Example 3 below). 

\section{Main Lemmas}

Let  $G$ be a platform group in a key exchanged scheme. Suppose that $G$ is a subset of a finite dimensional linear space $V.$ Two correspondents, Alice and Bob, agree about an element $h \in G$ and two finitely generated subgroups  $A$ and $B$      of  $G$ given by their finite generating sets.  Suppose that each element $a\in A$  commutes with every element $b \in B$. All these data are public. 

 Then the correspondences beginning wth $h$ repeatedly publish elements $\phi_{a_i,b_i}(u) = a_iub_i$, where $a_i, b_i \in A$  (Alice), and $\phi_{c_j,d_j}(u) = c_jud_j,$ where $c_j, d_j \in B$ (Bob), where  $u$ is one of the given or previously constructed elements. The exchanged key has the form
\begin{equation}
\label{eq:2}
K = \phi_{f_1, g_1}^{\epsilon_1}(\phi_{f_2, g_2}^{\epsilon_2}( ... (\phi_{f_t, g_t}^{\epsilon_t} (h) ... )),
\end{equation}
\noindent  where every pair  $(f_r,g_r)$ coincides either with a pair of the form  $(a_i,b_i)$, or with the pair of the form   $(c_j,d_j)$, $\epsilon_r\in \{\pm 1\}.$

The following lemma shows how we can efficiently construct bases of linear subspaces of  $V$, that are generated by elements of $G$ of the certain form. Different versions of this lemma  have been proved in  \cite{RomAC}, \cite{Roman} and   \cite{MR}.  
\begin{lemma}
\label{le:1}
Let   $A=$ gp($a_1, ..., a_k$) be a finitely generated subgroup of group $G$, that is a subset of a finite dimensional linear space $V$ over a field $\mathbb{F},$ and  $h$ be a fixed element of $G.$ Suppose that all main computations over $V$, i.e., addition, multiplication to scalar, can be efficiently done . Then each finite set of linear equations over $\mathbb{F}$ can be efficiently solved. Then we can efficiently construct a base  $E=\{e_1, ..., e_s\}$  of the linear subspace Lin($AhA$), generated by all elements of the form $ahb,$ где $a, b\in A.$
\end{lemma}
\begin{proof}
Consider the arbitrary ordered, beginning with $h,$ set of all elements of the form   $c^{\epsilon}hd^{\eta },$ where $\epsilon , \eta \in \{\pm{1}\},$ and  $c, d$ are elements of the form  $a_i$, or $1$. This set is called the first list and is denoted as  $L_1$.  The following operations will give a part $\{e_1, e_2, ...\}$ of the constructing base $E$. This part is a base of the linear subspace  Lin($L_1$):

1) Let $e_1 = h.$

2) Let the elements  $\{e_1, ..., e_t\}$ of the base  $E$ have been constructed.  We take the following element $c^{\epsilon}hd^{\eta }$ in $L_1.$  If it is linearly depended with the constructed elements we delete it. If it doesn't happen we add it to the set of constructed elements, i.e., it is included to $E.$  

When $L_1$ is over we form a new arbitrary ordered list $L_2$  consisting of all elements of the form  $c^{\epsilon}e_jd^{\nu}$, where $e_j$ is an element of the part of $E$, that has been constructed after $L_1$ ended (with exception $e_1$). 

Further, we consequently consider the elements of $L_2$ and operate as in  2). After $L_2$ is over, and we get a part of  $E$, that is a base of the linear subspace of $V$ generated by  $L_1$ and $L_2$, we construct the third list, and so on. 

3) The process ends when the operation with a list  $L_i$ doesn't add any new element of $E.$  

To explain the assertion 3) we note that every new list consists of the elements of the previous list multiplied in the both sides to the generating elements of   $A$, or to their inverses (one of these factors can be $1$).  Let  $X=\{a_1^{\pm 1}, ..., a_k^{\pm 1}, 1\}.$ Then  $L_1\subseteq XhX,$ $L_2\subseteq X^2hX^2,$ and so on. If after operation with $L_{i+1}\subseteq X^{i+1}hX^{i+1}$ no one element was added to $E,$ then  $L_{i+1}$ lies in the subspace generated by all previously constructed elements of $E,$ i.e.,  $X^{i+1}hX^{i+1}\subseteq  Lin(\cup_{j=1}^{i}X^jhX^j).$ Then  
$X^{i+2}hX^{i+2}\subseteq X(Lin(\cup_{j=1}^{i}X^jhX^j))X \subseteq Lin(\cup_{j=1}^{i+1}X^jhX^j)\subseteq Lin(\cup_{j=1}^{i}X^jhX^j).$ Hence operations with $L_{i+2}$ doesn't add new elements to $E.$ It is clear, that the considering subspace coincides with the subspace generated by the all lists. It follows that the number of all lists that add new elements to $E$ doesn't exceed the dimension of $V.$  

\end{proof}

The following lemma is a key statement for the forthcoming cryptanalysis.  We suppose that all given above agreements are satisfied. 
\begin{lemma}
\label{le:2}
Let $G$ be a group, that is a subset of a finite dimensional linear space  $V$ over 
a field $\mathbb{F}.$ Assume that all assumptions about main operations over $V$ given in the Lemma \ref{le:1} are satisfied. Let $$v =\phi_{a,b}(u),$$
\noindent
where  $a,b \in A$  are Alice's private parameters (similar statement is true for Bob's private parameters). 

Then for every element of the form $w = \phi_{c,d}(u),$ where  $c, d \in B$  
(in other words $w\in BuB$), we can efficiently construct 
 $z = \phi_{a,b}(w)$  based on the structure of $V.$ 
\end{lemma}
\begin{proof}
Obviously $v \in AuA.$ Let  $E=\{a_1ub_1, ..., a_rub_r\}, a_i, b_i \in A,$ be a basis of Lin($AuA$), that is efficiently obtained by Lemma  \ref{le:1}. By the Gauss elimination process we get the unique expression 
\begin{equation}
\label{eq:3}
v = \sum_{i=1}^r\alpha_ia_iub_i, \   \alpha_i \in \mathbb{F}.
\end{equation}
All the values in the right hand side of (\ref{eq:3}) are known now. We substitute  to the right hand side of  (\ref{eq:3}) $w$ instead of $u$. Since elements of $A$ and $B$ are pairwise commuting we obtain
$$
\sum_{i=1}^r\alpha_ia_iwb_i = \sum_{i=1}^r\alpha_ia_icudb_i = c(\sum_{i=1}^r\alpha_ia_iub_i)d 
$$
\begin{equation}
\label{eq:4}
= cvd=caubd =a(cud)b=awb=z.
\end{equation}
\end{proof}

Now we formulate a mnemonic rule of an efficient constructing  of a specific element following by the Lemmas \ref{le:1} and  \ref{le:2}:
$$
v=\phi_{a,b}(u) \  (a, b \in A)\   \& \   w\in BuB \   \Rightarrow \    \phi_{a,b}(w);
$$
\begin{equation}
\label{eq:5}
v=\phi_{c,d}(u) \ (c,d \in B) \   \& \   w\in AuA \ \Rightarrow  \  \phi_{c,d}(w).
\end{equation}

It means that if we find a basis of the underlined linear subspace by the Lemma  \ref{le:1} then  from the elements  $u$ and $v$ in the left hand side of the corresponding part of the rule, we can efficiently construct the image of  $w$ in the right hand side of the rule.

\section{Examples}

\paragraph{\bf Example 1.}
We describe the protocol 1 by Wang et al. \cite{WW}. In this protocol one of the Artin braid groups  is proposed as the platform. 

Let $B_n$ denote the Artin braid group on $n$ strings, $n \in \mathbb{N}.$ R. Lawrence described in 1990 a family of so called Lawrence representations of $B_n.$ Around 2001 S. Bigelow \cite{Big} and D. Krammer \cite{Kr} independently proved that all braid groups $B_n$ are linear. Their work used the  Lawrence-Krammer representations $\rho_n : B_n \rightarrow GL_{n(n-1)/2}(\mathbb{Z}[t^{\pm 1}, s^{\pm 1}])$ that   has been proved faithful for every $n\in \mathbb{N}.$  One can effectively find the image $\rho_n (g)$ for every element $g \in B_n.$ 
Moreover, there exists an effective procedure to recover a braid $g \in B_n$ from its image $\rho_n (g).$ It was shown in \cite{CJ} that it can be done in $O(2m^3log d_t)$ multiplications of entries in $\rho_n (g).$ Here $m = n(n-1)/2$ and $d_t$ is a parameter  that can be effectively computed by $\rho_n (g).$  See \cite{CJ} for details.  

 Thus we can assume that the platform  $G$ is a part of a finite dimensional linear space $V.$

Alice and Bob agree about  group $G$ and a random  element  $h\in G$, as well as about  two finitely generated  subgroups    $ A $ and $B $  such that   $ab=ba$ for every pair   $a \in A$ and $b \in B.$ These data are public. 

The algorithm  works as follows.

Alice chooses four elements: $c_1,c_2,d_1,d_2   \in  A,$ then computes and publishes   
$x = d_1 c_1hc_2 d_2$ for Bob.

Bob chooses six elements:  $f_1,f_2,g_1,g_2,g_3,g_4 \in B,$ then computes and publishes  $y = g_1 f_1hf_2 g_2$ and $w = g_3 f_1xf_2 g_4,$ for Alice. 
 
Alice picks up two elements:  $d_3,d_4 \in A,$ then computes and publishes  
 $z = d_3 c_1yc_2 d_4$ and $u = d_1^{-1}wd_2^{-1},$ for Bob.

Bob computes and publishes   $v = g_1^{-1}zg_2^{-1}$ for Alice.  

Alice computes the key $K_A = d_3^{-1}vd_4^{-1}=c_1 f_1hf_2 c_2.$

Bob computes the key  $K_B = g_3^{-1}ug_4^{-1} = c_1 f_1hf_2 c_2$, that is equal to $K_A.$  

Now Alice and Bob have the common secret key  $K = K_A = K_B$.  

\paragraph
{\bf Cryptanalysis.}

The following transformations were used in the protocol:
\begin{equation}
\label{eq:6}
\phi_{d_1c_1,c_2d_2}, \phi_{g_1f_1,f_2g_2}, \phi_{g_3f_1, f_2g_4}, \phi_{d_3c_1, c_2d_4}, \phi_{d_1,d_2}^{-1},
\phi_{g_1,g_2}^{-1}.
\end{equation}  

\bigskip
By direct computation we get an expression  of   $K$: 
\begin{equation}
\label{eq:7}
K = \phi_{c_1f_1, f_2c_2}(h) =  \phi_{d_1,d_2}^{-1}(\phi_{d_1c_1,c_2d_2}(\phi_{g_1,g_2}^{-1}(\phi_{g_1f_1,f_2g_2}(h)))).
\end{equation}
We are going to show that the key $K$ can be efficiently obtained by Lemmas  \ref{le:1}  and \ref{le:2}.  

The output of the first transformation  $y=\phi_{g_1f_1,f_2g_2}(h)$ is public. 

 The output of the second transformation  $\phi_{g_1,g_2}^{-1}(y)$ can be efficiently obtained by the mnemonic rule:  $$v= \phi_{g_1,g_2}^{-1}(z) \ \&
\  y \in AzA \ \Rightarrow \  \phi_{g_1,g_2}^{-1}(y)=f_1hf_2.$$

The output of the third transformation can be efficiently obtained by the mnemonic rule:   $$x= \phi_{d_1c_1,c_2d_2}(h) \ \&
\  f_1hf_2 \in BhB \ \Rightarrow \  \phi_{d_1c_1,c_2d_2}(f_1hf_2)=d_1c_1f_1hf_2c_2d_2.$$

The output of the fourth transformation can be efficiently obtained by the mnemonic rule:  $$u=\phi_{d_1,d_2}^{-1}(w) \ \&
\  d_1c_1f_1hf_2c_2d_2 \in BwB \ \Rightarrow \  \phi_{d_1,d_2}^{-1}(d_1c_1f_1hf_2c_2d_2)$$
$$=c_1 f_1hf_2 c_2=K.$$
Thus we have $K.$
\paragraph{\bf Example 2.} Well known protocol by Ko et al.  \cite{KL} usually is called  {\it noncommutative analog of Diffie-Hellman protocol}. The authors of \cite{KL} proposed one of the Artin braid groups  $B_n,$ $ n \in \mathbb{N},$  as a platform. On the matrix representation of $B_n$ see the previous example. 

The agreements about  $G=B_n$,  $h \in G,$ and   $A$ and $B$  are the same as in the previous example.

Now the algorithm works as follows. 

Alice chooses randomly  $a  \in  A,$ computes and publishes   
$h^a=aha^{-1}$ for Bob.

Bob picks up randomly $b  \in  B,$ computes and publishes   
$h^b=bhb^{-1}$ for Alice.

Alice computes the key $K_A = (h^b)^a = h^{ab}.$

Bob computes the key $K_B = (h^a)^b = h^{ba}.$

Since $ab=ba$  they get the exchanged secret key $K=K_A = K_B.$

\paragraph
{\bf Cryptanalysis.}

We see that 
\begin{equation}
\label{eq:8}
K = abha^{-1}b^{-1} =\phi_{a,a^{-1}}( \phi_{b,b^{-1}}(h)).
\end{equation} 

The output of the first transformation  $h^b =\phi_{b,b^{-1}}(h)$ is public.  

The output of the second transformation can be efficiently obtained by the mnemonic rule:
$$h^a = \phi_{a,a^{-1}}(h)\  \& \  h^b \in BhB \  \Rightarrow \   \phi_{a,a^{-1}}(h^b) = K.$$
Thus we have $K.$

\paragraph{\bf Example 3.} We describe the protocol proposed by B. and T. Harley \cite{HH}, \cite{H}. 
Let $G$ be a finitely generated commutative subgroup of the general linear group    GL$_n(\mathbb{F})$ over a field $\mathbb{F}.$ These data are public.

The algorithm works as follows.

Bob chooses randomly $y \in \mathbb{F}^n$ and $b \in G$, computes and publishes    $yb$.

Alice wants to send a message  $x \in \mathbb{F}^n$ to Bob. She chooses randomly     $a_1, a \in G,$  then computes and publishes    $(xa, yba_1)$ for 
Bob.

Bob chooses randomly  $b_1, b_2 \in G$, then computes and publishes   $(xab_1, ya_1b_2),$ for Alice. 

Alice computes   $(xb_1, yb_2)$ and then publishes  $xb_1-yb_2$ for Bob.

Bob computes  $x - yb_2b_1^{-1} $ and then recovers $x.$

Bob can use    $yb$ in forthcoming sessions.

\paragraph
{\bf Cryptanalysis.} 

Since $G$ is commutative we can assume that the correspondences use arbitrary right hand side multiplications   $\rho_c = \phi_{1,c}, c \in G$, and that $A=B=G.$   The arguments of the Lemma \ref{le:1} allow to construct bases of subspaces of the form  Lin($gG$), $g \in G$.  The assumption  $w\in BuB$  of the Lemma  \ref{le:2} automatically done.  Now the mnemonic rule is simpler.  The group  $G$ generates a finite dimensional subspace Lin$(G)$ of  GL$_n(\mathbb{F})$. Every $\rho_c$  is uniquely extended to Lin$(G)$. It follows that the secret element in the protocol can be efficiently computed as in the previous examples.    

In the protocol, the public  elements  $yb, xa$  and $xb_1 - yb_2$ are given, the transformations  $\rho_{b_1}, \rho_{a_1}, \rho_{a_1b_2b^{-1}}$ have been used. The first transformation corresponds to $(xa, xab_1)$,   the second transformation corresponds to ($yb, yba_1)$, and the third transformation corresponds to  $(yb, ya_1b_2)$.

Then we have:
\begin{equation}
\label{eq:7}
x = \rho_{b_1}^{-1}(xb_1-yb_2)+ \rho_{b_1}^{-1}(\rho_{a_1}^{-1}(ya_1b_2)).
\end{equation}
Both terms are efficiently computed by the mnemonic rule. 

\section{Complexity of the proposed algorithms}

The algorithm described in the Lemma \ref{le:1} gives a base of the subspace of the given finitely dimensional linear space.  It can be done by the Gauss elimination process.  In every case we need to determine existing of a solution only.  Note that the Gauss algorithm runs for matrix of size   $(t\times s)$ in  $O(t^2s)$ steps. Let $r$ be the dimension of  $V.$ Then $r$ is a number of equations in each of the considering systems of linear equations. The number of variables doesn't exceed  $r$ because it is equal to the number previously included basic elements.  Hence, every time we use  not more than  $O(r^3)$ operations. The total number of considering lists doesn't exceed  $r$ because every such list adds at least one new basic elements to the constructing base.  Every list contains not more than   $4k^2r$ elements. The total number of such elements doesn't exceed  $2k^2r^2$. Hence we have (very crude) an upper bound  $O(k^2r^5)$ of such operations. 

In some cases we need to estimate a number of operations giving some previous representation of the platform group $G$ by matrices, as well as a number of operations needed to compute the inverse map to $G.$ For example, in 
 \cite{CJ} was shown that the standard form of any element $g$ of $B_n$ can be recovered with its linear image by not more than   $O(n^3log_2d_t)$ operations. Here   $d_t$ is some efficiently computed parameter depending from $g.$ Note, that the linear image of $g$ is determined in linear time with respect to the length of $g.$

In the algorithm of the Lemma \ref{le:2} we use the Gauss elimination process. In the case, this process finds the unique solution. Since a total number of base constructions is not more than the total number of  computations in this base,  the estimation keeps its form:  $O(k^2r^5).$

\end{document}